\documentclass[letterpaper, 10 pt, conference]{ieeeconf}  

\IEEEoverridecommandlockouts                              
\usepackage{bm}
\usepackage{framed}
\usepackage{xfrac}
\usepackage{amsmath}
\usepackage{amsfonts}
\usepackage{graphicx}
\usepackage{wrapfig}
\usepackage{subfigure} 
\newtheorem{definition}{Definition}
\newtheorem{lemma}{Lemma}
\newtheorem{assumption}{Assumption}
\newtheorem{theorem}{Theorem}
\usepackage{algorithm}
\usepackage{algorithmic}
\usepackage{amssymb}
\usepackage{amsmath}

\title{\LARGE \bf Fast, Accurate Second Order Methods for Network Optimization\thanks{This research was supported in parts by the AFOSR Complex Networks Program and ONR Basic Research Challenge Program in Decentralized and Online }
}

\author{Rasul Tutunov, Haitham Bou Ammar, and Ali Jadbabaie
\thanks{R. Tutunov is with the Computer and Information Science Department, University of Pennsylvania
        {\tt\small tutunov@seas.upenn.edu}} 
\thanks{H. Bou Ammar is with the Computer and Information Science Department, University of Pennsylvania
        {\tt\small haithamb@seas.upenn.edu}}        
\thanks{A. Jadbabaie is with the Department of Electrical Engineering, University of Pennsylvania 
        {\tt\small jadbabai@seas.upenn.edu}}%
}

\begin{document}

\maketitle
\thispagestyle{empty}
\pagestyle{empty}

\begin{abstract}
Dual descent methods are commonly used to solve network flow optimization problems, since their implementation can be distributed over the network. These algorithms, however, often exhibit slow convergence rates. Approximate Newton methods which compute descent directions locally have been proposed as alternatives to accelerate the convergence rates of conventional dual descent. The effectiveness of these methods, is limited by the accuracy of such approximations. In this paper, we propose an efficient and accurate distributed second order method for network flow problems. The proposed approach utilizes the sparsity pattern of the dual Hessian to approximate the  the Newton direction using a novel distributed solver for symmetric diagonally dominant linear equations.  Our solver is based on a distributed implementation of a recent parallel solver of Spielman and Peng (2014).
 We analyze the properties of the proposed algorithm and show that, similar to conventional Newton methods, superlinear convergence within a neighborhood of the optimal value is attained.  We finally demonstrate the effectiveness of the approach in a set of experiments on randomly generated networks.
\end{abstract}

\section{INTRODUCTION}
Conventional methods for distributed network optimization are based on sub-gradient descent in either the primal or dual domains, see~\cite{c8, c9, c10, c13}. For a large class of problems, these techniques yield iterations that can be implemented in a distributed fashion by only using local information. Their applicability, however, is limited by increasingly \emph{slow} convergence rates. Second order Newton methods~\cite{c3,c4} are known to overcome this limitation leading to improved convergence rates. 

Unfortunately, computing \emph{exact} Newton directions based only on local information is challenging. Specifically, to determine the Newton direction, the inverse of the dual Hessian is needed. Determining this inverse, however, requires global information. Consequently, authors in~\cite{c5,c6} proposed approximate algorithms for determining these Newton iterates in a distributed fashion. Accelerated Dual Descent (ADD)~\cite{c6}, for instance, exploits the fact that the dual Hessian is the weighted Laplacian of the network and performs a truncated Neumann expansion of the inverse to determine a local approximate to the exact direction. ADD allows for a tradeoff between accurate Hessian approximations and communication costs through the N-Hop design, where increased N allows for more accurate inverse approximations arriving at increased cost, and  lower values of N reduce accuracy but improve computational times. Though successful, the effectiveness of these approaches highly depend on the accuracy of the truncated Hessian inverse which is used to approximate the Newton direction. As shown in Section~\ref{Sec:Exp}, the approximated iterate can resemble high variation to the real Newton direction, decreasing the applicability of these techniques. 

Exploiting the sparsity pattern of the dual Hessian, in this paper we tackle the above problem and propose a Newton method for network optimization that is both faster and more accurate. Using recently-developed solvers for symmetric diagonally dominant (SDDM) linear equations, we approximate the Newton direction up-to any arbitrary precision $\epsilon > 0$. The solver is a distributed implementation of ~\cite{c11} constructing what is known as an inverse chain. We analyze the properties of the proposed algorithm and show that, similar to conventional Newton methods, superlinear convergence within a neighborhood of the optimal value is attained. We finally demonstrate the effectiveness of the approach in a set of experiments on randomly generated networks. Namely, we show that our method is capable of significantly outperforming state-of-the-art methods in both the convergence speeds and in the accuracy of approximating the Newton direction. 

The remainder of the paper is organized as follows. Section~\ref{Sec:Back} draws upon background material needed for the remainder of the paper. Section~\ref{Sec:ProbDefS} defines the network flow optimization problem targeted in this paper. Section~\ref{Sec:SDDM} details our proposed distributed solver for SDDM linear systems. Section~\ref{Sec:NewtonApprox} introduces the approximate Newton method and rigorously analyzes its theoretical guarantees. Section~\ref{Sec:Exp} presents the experimental results. Finally, Section~\ref{Sec:Con} concludes pointing-out interesting directions for future research.

\section{BACKGROUND}\label{Sec:Back}
\subsection{SDDM Linear Systems}\label{Sec:ProbDef}
To determine the Newton direction, we need to solve a symmetric diagonally dominant system of linear equations, defined as: 
\begin{equation}\label{lin_sys}
\bm{M}_{0}\bm{x} = \bm{b}_{0}
\end{equation}
where $\bm{M}_{0}$ is a Symmetric Diagonally Dominant M-Matrix (SDDM). Namely, $\bm{M}_{0}$ is symmetric positive definite with non-positive off diagonal elements, such that for all $i=1,2,\ldots, n$:
\begin{equation*}
\left[\bm{M}_{0}\right]_{ii} \ge -\sum_{j=1, j\ne i}^{n}\left[\bm{M}_{0}\right]_{ij}
\end{equation*}
The system of Equations in~\ref{lin_sys} can be interpreted as representing an undirected weighted graph, $\mathcal{G}$, with $\bm{M}_{0}$ being its Laplacian. Namely, $\mathcal{G} = \left(\mathcal{N},\mathcal{E},\bm{W}\right)$, with $\mathcal{N}$ representing the set of nodes, $\mathcal{E}$ denoting the edges, and $\bm{W}$ representing the weighted graph adjacency. Nodes $\bm{v}_i$ and $\bm{v}_j$ are connected with an edge $\bm{e}=\left(i,j\right)$ iff $\bm{W}_{ij}> 0$, where: 
\begin{equation*}
\bm{W}_{ij} = \left[\bm{M}_{0}\right]_{ii} \ \ \ \text{(if $i = j$)}, \ \ \ \ \text{or} \ \ \ \  \bm{W}_{ij} = -\left[\bm{M}_{0}\right]_{ij}, \text{otherwise}. 
\end{equation*}
Following~\cite{c11}, we seek $\epsilon$-approximate solutions to $\bm{x}^{\star}$, being the exact solution of $\bm{M}_{0}\bm{x}=\bm{b}_{0}$, defined as:
\begin{definition}
Let $\bm{x}^{\star}\in \mathbb{R}^{n}$ be the solution of $\bm{M}\bm{x}=\bm{b}_{0}$. A vector $\tilde{\bm{x}}\in \mathbb{R}^{n}$ is called an $\epsilon-$ approximate solution, if:
\begin{equation}
\left|\left|\bm{x}^{\star} - \tilde{\bm{x}}\right|\right|_{\bm{M}_{0}} \le \epsilon\left|\left|\bm{x}^{\star}\right|\right|_{\bm{M}_{0}}, \ \ \ \text{where $\left|\left|\bm{u}\right|\right|^{2}_{\bm{M}_0} = \bm{u}^{\mathsf{T}}\bm{M}_{0}\bm{u}$.
}
\end{equation}
\end{definition}

The R-hop neighbourhood of node $\bm{v}_{k}$ is defined as $\mathbb{N}_{r}\left(\bm{v}_{k}\right) = \{\bm{v}\in \mathcal{N}: \text{dist}\left(\bm{v}_{k}, \bm{v}\right)\le r\}$. We also make use of the diameter of a graph, $\mathcal{G}$, defined as $\text{diam}\left(\mathcal{G}\right) = \max_{\bm{v}_{i},\bm{v}_{j}\in \mathcal{N}}\text{dist}\left(\bm{v}_i,\bm{v}_j\right)$. 

\begin{definition}
A matrix $\bm{A} \in \mathbb{R}^{n\times n}$ is said to have a sparsity pattern corresponding to the R-hop neighborhood if $\bm{A}_{ij} = 0$ for all $i = 1,\ldots, n$ and for all $j$ such that $\bm{v}_j\notin \mathbb{N}_{r}\left(\bm{v}_i\right)$. 
\end{definition}

We will denote the spectral radius of a matrix $\bm{A}$ by $\rho\left(\bm{A}\right) = \max{\left|\bm{\lambda}_i\right|}$, where $\bm{\lambda}_i$ represents an eigenvalue of the matrix $\bm{A}$. Furthermore, we will make use of the condition number\footnote{Please note that in the case of the graph Laplacian,  the condition number is defined as the ratio of the largest to the smallest nonzero eigenvalues.}, $\kappa\left(\bm{A}\right)$ of a matrix $\bm{A}$ defined as $\kappa=\left|\frac{\bm{\lambda}_{\text{max}}\left(\bm{A}\right)}{\bm{\lambda}_{\text{min}}\left(\bm{A}\right)}\right|$.  In~\cite{DaanS} it is shown that the condition number of the graph Laplacian is at most $\mathcal{O}\left(n^{3}\frac{\bm{W}_{\text{max}}}{\bm{W}_{\text{min}}}\right)$, where $\bm{W}_{\text{max}}$ and $\bm{W}_{\text{min}}$ represent the largest and the smallest edge weights in $\mathcal{G}$. Finally, the condition number of a sub-matrix of the Laplacian is at most $\mathcal{O}\left(n^4\frac{\bm{W}_{\text{max}}}{\bm{W}_{\text{min}}}\right)$, see~\cite{c11}.  

\subsection{Standard Splittings \& Approximations}\label{Sec:Standard}
For determining the Newton direction, we propose a fast distributed solver for symmetric diagonally dominant linear equations. Our approach is based on a distributed implementation of the parallel solver of Spielman and Peng~\cite{c11}. Before detailing the parallel solver, however, we next provide basic notions and notations required.

\begin{definition}
The standard splitting of a symmetric matrix $\bm{M}_{0}$ is: 
\begin{equation}
\bm{M}_{0} = \bm{D}_{0} - \bm{A}_{0}.
\end{equation}
Here, $\bm{D}_0$ is a diagonal matrix such that $\left[\bm{D}_{0}\right]_{ii} = \left[\bm{M}_{0}\right]_{ii}$ for $i = 1,2,\ldots,n$, and $\bm{A}_0$ representing a non-negative symmetric matrix such that $\left[\bm{A}_0\right]_{ij} = -\left[\bm{M}_{0}\right]_{ij}$ if $i\ne j$, and $\left[\bm{A}_{0}\right]_{ii} = 0$.
\end{definition}
We also define the Loewner ordering: 
\begin{definition}
Let $\mathcal{\bm{S}}_{(n)}$ be the space of $n \times n$-symmetric matrices. The Loewner ordering $\preceq$ is a partial order on $\mathcal{\bm{S}}_{(n)}$ such that $\bm{Y}\preceq \bm{X}$ if and only if $\bm{X} - \bm{Y}$ is positive semidefinite.
\end{definition}

Finally, we define the ``$\approx_{\alpha}$'' operation used in the sequel to come as:
\begin{definition}
Let $\bm{X}$ and $\bm{Y}$ be positive semidefinite symmetric matrices. Then $\bm{X}\approx_{\alpha} \bm{Y}$ if and only iff
\begin{equation}
e^{-\alpha}\bm{X} \preceq \bm{Y} \preceq e^{\alpha}\bm{X}
\end{equation}
with $\bm{A}\preceq \bm{B}$ meaning $\bm{B} - \bm{A}$ is positive semidefinite.
\end{definition}

Based on the above definitions, the following lemma represents the basic characteristics of the $\approx_{\alpha}$ operator:
\begin{lemma}~\cite{c11}\label{approx_lemma_facts}
Let $\bm{X},\bm{Y},\bm{Z}$ and, $\bm{Q}$ be symmetric positive semi definite matrices. Then
\begin{enumerate}
\item[] (1) If $\bm{X}\approx_{\alpha} \bm{Y}$, then $\bm{X} + \bm{Z} \approx_{\alpha} \bm{Y} + \bm{Z}$, (2) If $\bm{X}\approx_{\alpha} \bm{Y}$ and $\bm{Z}\approx_{\alpha} \bm{Q}$, then $\bm{X} + \bm{Z} \approx_{\alpha} \bm{Y} + \bm{Q}$
\item[] (3) If $\bm{X}\approx_{\alpha} \bm{Y}$ and $\bm{Z}\approx_{\alpha} \bm{Q}$, then $\bm{X} + \bm{Z} \approx_{\alpha} \bm{Y} + \bm{Q}$, (4) If $\bm{X}\approx_{\alpha_1} \bm{Y}$ and $\bm{Y} \approx_{\alpha_2} \bm{Z}$, then $\bm{X} \approx_{\alpha_1 + \alpha_2} \bm{Z}$
\item[] (5) If $\bm{X}$, and $\bm{Y}$ are non singular and $\bm{X}\approx_{\alpha} \bm{Y}$, then $\bm{X}^{-1}\approx_{\alpha} \bm{Y}^{-1}$, (6) If $\bm{X}\approx_{\alpha} \bm{Y}$ and $\bm{V}$ is a matrix, then $\bm{V}^{\mathsf{T}}\bm{X}\bm{V}\approx_{\alpha}\bm{V}^{\mathsf{T}}\bm{Y}\bm{V}$
\end{enumerate}
\end{lemma}

The next lemma shows that good approximations of $\bm{M}^{-1}_0$ guarantee good approximated solutions of $\bm{M}_{0}\bm{x}=\bm{b}_{0}$.


\begin{lemma}\label{lemma_approx_matrix_inverse}
Let $\bm{Z}_0\approx_{\epsilon}\bm{M}^{-1}_0$, and $\tilde{\bm{x}} = \bm{Z}_0\bm{b}_0$. Then $\tilde{\bm{x}}$ is $\sqrt{2^{\epsilon}(e^{\epsilon} - 1)}$ approximate solution of $\bm{M}_{0}\bm{x}=\bm{b}_{0}$.
\end{lemma}
\begin{proof}
The proof can be found in the appendix. 
\end{proof}

\subsection{The Parallel SDDM Solver}\label{Parallel:SDDM}\label{Sec:ParrallelSolver}
The parallel SDDM solver proposed in~\cite{c11} is a parallelized technique for solving the problem of Section~\ref{Sec:ProbDef}. It makes use of inverse approximated chains (see Definition~\ref{Def:InvChain}) to determine $\tilde{\bm{x}}$ and can be split in two steps. In the first step, denoted as Algorithm~\ref{Algo:Inv}, a ``crude'' approximation, $\bm{x}_{0}$, of $\bm{\tilde{x}}$ is returned. $\bm{x}_{0}$ is driven to the $\epsilon$-close solution, $\tilde{\bm{x}}$, using Richardson Preconditioning in Algorithm~\ref{Algo:Inv2}. Before we proceed, we start with the following two Lemmas which enable the definition of inverse chain approximation. 

\begin{lemma}~\cite{c11}\label{SDDM_splitting_lemma}
If $\bm{M} = \bm{D} - \bm{A}$ is an SDDM matrix, with $\bm{D}$ being positive diagonal, and $\bm{A}$ denoting a non-negative symmetric matrix, then $\bm{D} - \bm{A}\bm{D}^{-1}\bm{A}$ is also SDDM.
\end{lemma}

\begin{lemma}~\cite{c11}\label{approx_inverse_formulae_lemma}
Let $\bm{M} =\bm{D} - \bm{A}$ be an SDDM matrix, where $\bm{D}$ is positive diagonal and, $\bm{A}$ a symmetric matrix. Then 
\begin{align}\label{Inv_of_SDDM}
\left(\bm{D}-\bm{A}\right)^{-1} &= \frac{1}{2}\Big[\bm{D}^{-1} + \left(\bm{I} + \bm{D}^{-1}\bm{A}\right)\left(\bm{D} - \bm{A}\bm{D}^{-1}\bm{A}\right)^{-1} \\ \nonumber 
&\hspace{13em} \left(\bm{I} + \bm{A}\bm{D}^{-1}\right)\Big].
\end{align}
\end{lemma}


Given the results in Lemmas \ref{SDDM_splitting_lemma} and \ref{approx_inverse_formulae_lemma}, we now can consider inverse approximated chains of $\bm{M}_0$:
\begin{definition}\label{Def:InvChain}
Let $\mathcal{C} = \{\bm{M}_0, \bm{M}_1, \ldots, \bm{M}_d\}$  be a collection of SDDM matrices such that $\bm{M}_i = \bm{D}_i - \bm{A}_i$, with $\bm{D}_i$ a positive diagonal matrix, and $\bm{A}_i$ denoting a non-negative symmetric matrix. Then $\mathcal{C}$ is an inverse approximated chain if there exists  positive real numbers $\epsilon_0, \epsilon_1, \ldots, \epsilon_d$ such that: (1) For $i = 1,\ldots, d$: $\bm{D}_i - \bm{A}_i \approx_{e_{i-1}} \bm{D}_{i-1} - \bm{A}_{i-1}\bm{D}^{-1}_{i-1}\bm{A}_{i-1}$, (2) $\bm{D}_i \approx_{\epsilon_{i-1}}\bm{D}_{i-1}$, and (3) $\bm{D}_d\approx_{\epsilon_d} \bm{D}_d - \bm{A}_d$. 
\end{definition}
 

\begin{algorithm}\label{Algo:CrudeParallel}
  \caption{$\text{ParallelRSolve}\left(\bm{M}_0,\bm{M}_1,\ldots, \bm{M}_d, \bm{b}_0\right)$}
  \label{Algo:Inv}
  \begin{algorithmic}[1]
	\STATE \textbf{Input}: Inverse approximated chain, $\{\bm{M}_0,\bm{M}_1,\ldots, \bm{M}_d\}$, and $\bm{b}_0$ being 
	
	\STATE \textbf{Output}: The ``crude'' approximation, $\bm{x}_0$, of $\bm{x}^{\star}$    
    \FOR  {$i=1$ to $d$} 
    	\STATE $\bm{b}_i = \left(\bm{I}+\bm{A}_{i-1}\bm{D}^{-1}_{i-1}\right)\bm{b}_{i-1}$
    \ENDFOR 
    \STATE $\bm{x}_d = \bm{D}^{-1}_d\bm{b}_d$
    \FOR  {$i=d-1$ to $0$} 
    	\STATE $\bm{x}_{i} = \frac{1}{2}\left[\bm{D}^{-1}_i\bm{b}_{i} + \left(\bm{I} + \bm{D}^{-1}_i\bm{A}_i\right)\bm{x}_{i+1}\right]$
    \ENDFOR 
    \STATE \textbf{return} $x_0$ 
  \end{algorithmic}
\end{algorithm}

The quality of the ``crude'' solution returned by Algorithm~\ref{Algo:Inv} is quantified in the following lemma:
\begin{lemma}~\cite{c11}\label{Rude_Alg_guarantee_Lemma}
Let $\{\bm{M}_0, \bm{M}_1,\ldots, \bm{M}_d\}$ be the inverse approximated chain and denote $\bm{Z}_0$ be the operator defined by $\text{ParallelRSolve}\left(\bm{M}_0, \bm{M}_1,\ldots, \bm{M}_d, \bm{b}_0\right)$, namely, $\bm{x}_0 = \bm{Z}_0\bm{b}_0$. Then
\begin{equation}\label{appr_inv_express}
\bm{Z}_0\approx_{\sum_{i=0}^{d}\epsilon_i}\bm{M}^{-1}_0
\end{equation}
\end{lemma}


Algorithm~\ref{Algo:Inv} returns a ``crude'' solution to $\bm{M}_{0}\bm{x}=\bm{b}$. To obtain arbitrary close solutions, Spielman {\it et. al}~\cite{c11} introduced the \emph{preconditioned Richardson iterative scheme}, summarized in Algorithm~\ref{Algo:Inv2}. Following their analysis, Lemma~\ref{Exact_Alg_guarantee_lemma} provides the iteration count needed by Algorithm~\ref{Algo:Inv2} to arrive at $\tilde{\bm{x}}$. 
\begin{algorithm}[h!]
  \caption{$\text{ParallelESolve}\left(\bm{M}_0, \bm{M}_1,\ldots, \bm{M}_d,  \bm{b}_0, \epsilon\right)$}
  \label{Algo:Inv2}
  \begin{algorithmic}[1]
	\STATE \textbf{Input}: Inverse approximated chain $\{\bm{M}_0,\bm{M}_1,\ldots, \bm{M}_d\}$, $\bm{b}_0$, and $\epsilon$. 
	\STATE \textbf{Output}: $\epsilon$ close approximation, $\tilde{\bm{x}}$, of $\bm{x}^*$
	\STATE \textbf{Initialize}: $\bm{y}_0 = 0$; \\
	$\chi = \text{ParallelRSolve}\left(\bm{M}_0,\bm{M}_1,\ldots, \bm{M}_d, \bm{b}_0\right)$ (i.e., Algorithm~\ref{Algo:Inv})
    \FOR  {$k=1$ to $q$}
    	\STATE $\bm{u}_{k}^{(1)} = \bm{M}_0\bm{y}_{k-1}$
    	\STATE $\bm{u}_{k}^{(2)} = \text{ParallelRSolve}\left(\bm{M}_0, \bm{M}_1,\ldots, \bm{M}_d, \bm{u}_{k}^{(1)}\right)$
    	\STATE $\bm{y}_{k} = \bm{y}_{k-1} - \bm{u}_{k}^{(2)} + \chi$ 
    \ENDFOR 
    \STATE $\tilde{\bm{x}} = \bm{y}_q$
    \STATE \textbf{return} $\tilde{\bm{x}}$ 
  \end{algorithmic}
\end{algorithm}

\begin{lemma}~\cite{c11}\label{Exact_Alg_guarantee_lemma}
Let $\{\bm{M}_0, \bm{M}_1\ldots \bm{M}_d\}$ be an inverse approximated chain such that $\sum_{i=1}^{d}\epsilon_i < \frac{1}{3}\ln2$. Then $\text{ParallelESolve}\left(\bm{M}_0, \bm{M}_1,\ldots, \bm{M}_d, \bm{b}_0, \epsilon\right)$ arrives at an $\epsilon$ close solution of $\bm{x}^{\star}$ in $q = \mathcal{O}\left(\log\frac{1}{\epsilon}\right)$ iterations. 
\end{lemma}


\section{NETWORK FLOW OPTIMIZATION}\label{Sec:ProbDefS}
We consider a network represented by a directed graph $\mathcal{G}=\left(\mathcal{N},\mathcal{E}\right)$ with node set $\mathcal{N}=\{1,\dots, N\}$ and edge set $\mathcal{E}=\{1,\dots, E\}$. The flow vector is denoted by $\bm{x}=\left[x^{(e)}\right]_{e\in\mathcal{E}}$, with $x^{(e)}$ representing the flow on edge $e$. The flow conservation conditions at nodes can be compactly represented as 
\begin{equation*}
\bm{A}\bm{x}=\bm{b},
\end{equation*}
where $\bm{A}$ is the $N \times E$ node-edge incidence matrix of $\mathcal{G}$ defined as 

\begin{displaymath}
   \bm{A}_{i,j} = \left\{
     \begin{array}{lr}
       1 & \text{if edge $j$ leaves node $i$} \\
      - 1 & \text{if edge $j$ enters node $i$} \\
      0 & \text{otherwise,}
     \end{array}
   \right.
\end{displaymath} 
and the vector $\bm{b} \in \bm{1}^{\perp}$ denotes the external source, i.e., $b^{(i)} > 0$ (or $b^{(i) } < 0$) indicates $b^{(i)}$ units of external flow enters (or leaves) node $i$. A cost function $\bm{\Phi}_{e}: \mathbb{R} \rightarrow \mathbb{R}$ is associated with each edge $e$. Namely, $\bm{\Phi}_{e}(x^{(e)})$ denotes the cost on edge $e$ as a function of the edge flow $x^{(e)}$. We assume that the cost functions $\bm{\Phi}_{e}$ are strictly convex and twice differentiable. Consequently, the minimum cost networks optimization problem can be written as 
\begin{align}
\label{Eq:OptimAll}
&\min_{\bm{x}} \sum_{e=1}^{E} \bm{\Phi}_{e}(\bm{x}^{(e)}) \\ \nonumber
&\text{s.t. $\bm{A}\bm{x}=\bm{b}$}
\end{align}

Our goal is to investigate Newton type methods for solving the problem in~\ref{Eq:OptimAll} in a distributed fashion. Before diving into these details, however, we next present basic ingredients needed for the remainder of the paper. 

\subsection{Dual Subgradient Method} 
The dual subgradient method optimizes the problem in Equation~\ref{Eq:OptimAll} by descending in the dual domain. The Lagrangian, $l: \mathbb{R}^{E} \times \mathbb{R}^{N} \rightarrow \mathbb{R}$ is given by 
\begin{equation*}
l (\bm{x},\bm{\lambda}) = -\sum_{e=1}^{E} \bm{\Phi}_{e}({x}^{(e)}) + \bm{\lambda}^{\mathsf{T}} (\bm{A}\bm{x}-\bm{b}). 
\end{equation*}

The dual function $q(\bm{\lambda})$ is then derived as 
\begin{align*}
q(\bm{\lambda}) &= \inf_{\bm{x} \in \mathbb{R}^{E}} l(\bm{x},\bm{\lambda}) \\
& = \inf_{\bm{x} \in \mathbb{R}^{E}} \left(-\sum_{e=1}^{E} \bm{\Phi}_{e} (x^{(e)}) + \bm{\lambda}^{\mathsf{T}}\bm{A}\bm{x}\right) - \bm{\lambda}^{\mathsf{T}}\bm{b} \\ 
& = \sum_{e=1}^{E} \inf_{x^{(e)} \in \mathbb{R}} \left(-\bm{\Phi}_{e}(x^{(e)}) + \left(\bm{\lambda}^{\mathsf{T}}\bm{A}\right)^{(e)}x^{(e)}\right) - \bm{\lambda}^{\mathsf{T}}\bm{b}.
\end{align*}
Hence, it can be clearly seen that the evaluation of the dual function $q(\bm{\lambda})$ decomposes into E one-dimensional optimization problems. We assume that each of these optimization problems have an optimal solution, which is unique by the strict convexity of the functions $\bm{\Phi}_{e}$. Denoting the solutions by $x^{(e)} (\bm{\lambda})$ and using the first order optimality conditions, it can be seen that for each edge, e, $x^{(e)}(\bm{\lambda})$ is given by\footnote{Note that if the dual is not continuously differentiable, the a generalized Hessian can be used.}
\begin{equation}
\label{Eq:MapBack}
x^{(e)}(\bm{\lambda})=[\dot{\bm{\Phi}}_{e}]^{-1}\left(\lambda^{(i)}-\lambda^{(j)}\right),
\end{equation}
where $i \in  \mathcal{N}$ and $j \in \mathcal{N}$ denote the source and destining nodes of edge $e=(i,j)$, respectively (see~\cite{c6} for details). Therefore, for an edge $e$, the evaluation of $x^{(e)}(\bm{\lambda})$ can be performed based on local information about the edge's cost function and the dual variables of the incident nodes, $i$ and $j$. 

The dual problem is defined as $\max_{\bm{\lambda}\in \mathbb{R}^{N}} q(\bm{\lambda})$. Since the dual function is convex, the optimization problem can be solved using gradient descent according to
\begin{equation}
\label{Eq:GD}
\bm{\lambda}_{k+1} = \bm{\lambda}_{k} - \alpha_{k}\bm{g}_{k} \hfill \ \ \text{for all $k \geq 0$,}
\end{equation}
with $k$ being the iteration index, and $\bm{g}_{k}=\bm{g}\left(\bm{\lambda}_{k}\right)=\nabla q(\bm{\lambda}_{k})$ denoting the gradient of the dual function evaluated at $\bm{\lambda}=\bm{\lambda}_{k}$. Importantly, the computation of the gradient can be performed as $\bm{g}_{k}=\bm{A}\bm{x}\left(\bm{\lambda}_{k}\right)-\bm{b}$, with $\bm{x}(\lambda_{k})$ being a vector composed of $x^{(e)}(\bm{\lambda}_{k})$ as determined by Equation~\ref{Eq:MapBack}. Further, due to the sparsity pattern of the incidence matrix $\bm{A}$, the $i^{th}$  element, $g_{k}^{(i)}$, of the gradient $\bm{g}_{k}$ can be computed as 
\begin{equation}
\label{Eq:GDDist}
g_{k}^{(i)}=\sum_{e=(i,j)} x^{(e)}(\bm{\lambda}_{k}) - \sum_{e=(j,i)} x^{(e)}(\bm{\lambda}_{k}) - b^{(i)}.
\end{equation}

Clearly, the algorithm in Equation~\ref{Eq:GD} can be implemented in a distributed fashion, where each node, $i$, maintains information about its dual, $\lambda_{k}^{(i)}$, and primal, $x^{(e)}(\bm{\lambda}_{k})$, iterates of the outgoing edges $e=(i,j)$. Gradient components can then be evaluated as per~\ref{Eq:GDDist} using only local information. Dual variables can then be updated using~\ref{Eq:GD}. Given the updated dual variables, the primal variables can be computed using~\ref{Eq:MapBack}. 

Although the distributed implementation avoids the cost and fragility of collecting all information at centralized location, practical applicability of gradient descent is hindered by slow convergence rates. This motivates the consideration of Newton methods discussed next. 

\subsection{Newton's Method for Dual Descent}
Newton's method is a descent algorithm along a scaled version of the gradient. Its iterates are typically given by 
\begin{equation}
\bm{\lambda}_{k+1} = \bm{\lambda}_{k}+\alpha_{k}\bm{d}_{k} \hfill \ \ \  \text{for all $k \geq 0$,}
\end{equation}
with $\bm{d}_{k}$ being the Newton direction at iteration $k$, and $\alpha_{k}$ denoting the step size. The Newton direction satisfies
\begin{equation}
\bm{H}_{k}\bm{d}_{k}=-\bm{g}_{k},
\end{equation} 
with $\bm{H}_{k}=\bm{H}(\bm{\lambda}_{k})=\nabla^{2}q(\bm{\lambda}_{k})$ being the Hessian of the dual function at the current iteration $k$.

\subsubsection{Properties of the Dual and Assumptions}
Here, we detail some assumptions needed by our approach. We also derive essential Lemmas quantifying properties of the dual Hessian. 
\begin{assumption}
The graph, $\mathcal{G}$, is connected, non-bipartite and has algebraic connectivity lower bound by a constant $\bm{\omega}$. 
\end{assumption}

\begin{assumption}\label{Ass:Two}
The cost functions, $\bm{\Phi}_{e}(\cdot)$, in Equation~\ref{Eq:OptimAll} are 
\begin{enumerate}
\item twice continuously differentiable satisfying 
\begin{equation*}
\gamma \leq \ddot{\bm{\Phi}}_{e}(\cdot) \leq \Gamma,
\end{equation*}
with $\gamma$ and $\Gamma$ are constants; and
\item Lipschitz Hessian invertible for all edges $e\in \mathcal{E}$
\begin{equation*}
\left|\frac{1}{\bm{\Phi}_{e}(\bm{x})} - \frac{1}{\bm{\Phi}_{e}(\hat{\bm{x}})}\right| \leq \bm{\delta}\left|\bm{x}-\hat{\bm{x}}\right|.
\end{equation*}
\end{enumerate}
\end{assumption}

The following two lemmas~\cite{c5,c6} quantify essential properties of the dual Hessian which we exploit through our algorithm to determine the approximate Newton direction. 

\begin{lemma}\label{lemma:Crap}
The dual objective $q(\bm{\lambda})=\bm{\lambda}^{\mathsf{T}}(\bm{A}\bm{x}(\bm{\lambda})-b)-\sum_{e}\bm{\Phi}_{e}(\bm{x}(\lambda))$ abides by the following two properties~\cite{Mike}:
\begin{enumerate}
\item The dual Hessian, $\bm{H}(\bm{\lambda})$, is a weighted Laplacian of $\mathcal{G}$:
\begin{equation*}
\bm{H}(\bm{\lambda})=\nabla^{2}q(\bm{\lambda})=\bm{A}\left[\nabla^{2}f(\bm{x}(\bm{\lambda}))\right]^{-1}\bm{A}^{\mathsf{T}}. 
\end{equation*}
\item The dual Hessian $\bm{H}(\bm{\lambda})$ is Lispshitz  continuous with respect to the Laplacian norm (i.e., $||\cdot||_{\mathcal{L}}$) where $\mathcal{L}$ is the unweighted laplacian satisfying $\mathcal{L}=\bm{A}\bm{A}^{\mathsf{T}}$ with $\bm{A}$ being the incidence matrix of $\mathcal{G}$. Namely, $\forall \bm{\lambda}, \bar{\bm{\lambda}}$: 
\begin{equation*}
||\bm{H}(\bar{\bm{\lambda}})-\bm{H}(\bm{\lambda})||_{\mathcal{L}} \leq B ||\bar{\bm{\lambda}}-\bm{\lambda}||_{\mathcal{L}},
\end{equation*}
with $B=\frac{\mu_{n}(\mathcal{L})\bm{\delta}}{\gamma \sqrt{\mu_{2}(\mathcal{L})}}$ where $\mu_{n}(\mathcal{L})$ and $\mu_{2}(\mathcal{L})$ denote the largest and second smallest eigenvalues of the Laplacian $\mathcal{L}$. 
\end{enumerate}
\end{lemma}
\begin{proof}
See Appendix.
\end{proof}
The following lemma follows from the above and is needed in the analysis later: 
\begin{lemma}\label{lemma:B}
If the dual Hessian $\bm{H}(\bm{\lambda})$ is Lipschitz continuous with respect to the Laplacian norm $||\cdot||_{\mathcal{L}}$ (i.e., Lemma~\ref{lemma:Crap}), then for any $\bm{\lambda}$ and $\hat{\bm{\lambda}}$ we have 
\begin{equation*}
||\nabla q(\hat{\bm{\lambda}})-\nabla q({\bm{\lambda}}) - \bm{H}(\bm{\lambda})(\hat{\lambda}-\bm{\lambda})||_{\mathcal{L}} \leq \frac{B}{2}||\hat{\bm{\lambda}} - \bm{\lambda}||_{\mathcal{L}}^{2}.
\end{equation*}
\end{lemma}
\begin{proof}
See Appendix. 
\end{proof}

As detailed in~\cite{c6}, the exact computation of the inverse of the Hessian needed for determining the Newton direction can not be attained exactly in a distributed fashion. Authors in~\cite{c5,c6} proposed approximation techniques for computing this direction. The effectiveness of these algorithms, however, highly depend on the accuracy of such an approximation. In this work, we propose a distributed approximator for the Newton direction capable of acquiring $\epsilon$-close solutions for any arbitrary $\epsilon$. Our results show that this new algorithm is capable of significantly surpassing others in literature where its performance accurately traces that of the standard centralized Newton approach. Next, we detail our distributed SDD solver being at the core of our approximator.  

\section{SDD DISTRIBUTED SOLVERS}\label{Sec:SDDM}
We propose a distributed solver for SDDM systems which can be used to determine an approximation to the Newton direction up to any arbitrary $\epsilon > 0$ (see Section~\ref{Sec:NewtonApprox}). Our method is based on a distributed implementation of the parallel solver of Section~\ref{Parallel:SDDM}. Similar to~\cite{c11}, we first introduce an approximate inverse chain which can be computed in a distributed fashion. This leads us to a distributed version of the ``crude'' solver (i.e., Algorithm~\ref{Algo:CrudeParallel}). Contrary to~\cite{c11}, however, we then generalize the ``crude'' distributed solver to acquire \emph{exact} solutions to an SDDM system. For a generic SDDM system of linear equations, our main results for determining an $\epsilon$-close solution (i.e., $||\tilde{\bm{x}}-\bm{x}^{*}||_{\bm{M}_{0}} \leq \epsilon ||\bm{x}^{*}||_{\bm{M}_{0}}$) is summarized by\footnote{The complete proofs can be found at https://db.tt/MbBW15Zx}: 

\begin{lemma}\label{lemma:SDDM}
For the system of equations represented by $\bm{M}_{0}\bm{x}=\bm{b}$, there is a distributed algorithm that uses only R-Hop information and computes the $\epsilon$-close solution, $\tilde{\bm{x}}$, in $T(n,\epsilon)=\mathcal{O}\left(\left(\frac{\beta\kappa(\bm{M}_{0})}{R}+\beta d_{\text{max}}R\right)\log\left(\frac{1}{\epsilon}\right)\right)$ time steps, with $\kappa(\bm{M}_{0})$ being the condition number of $\bm{M}_{0}$, $\beta=\min\left\{n,\frac{d_{\text{max}}^{R-1}-1}{d_{\text{max}}-1}\right\}$ representing the upper bound on the size of the R-Hop neighborhood, $d_{\text{max}}$ the maximal degree of $\mathcal{G}$, and $\epsilon \in (0,\frac{1}{2}]$ being the precision parameter. 
\end{lemma}

Analogous to~\cite{c11}, we will develop and analyze two distributed solvers for SDDM systems (i.e., ``crude'' R-Hop solver and ``exact'' R-Hop solver) leading to the proof of the above lemma. 

\subsection{``Crude'' R-Hop SDDM Solver}
Algorithm~\ref{Algo:DistRHop} presents the ``crude'' R-Hop solver for SDDM systems. Each node receives the $k^{th}$ row of $\bm{M}_{0}$ , $k^{th}$ component, $[\bm{b}_{0}]_{k}$ of $\bm{b}_{0}$, the length of the inverse chain, $d$, and the local communication bound\footnote{For simplicity, $R$ is assumed to be in the order of powers of 2, i.e., $R=2^{\rho}$.}  $R$ as inputs, and outputs the $k^{th}$ component of the ``rude'' approximation of $\bm{x}^{\star}$. 

\begin{algorithm}[h!]
  \caption{$\text{RDistRSolve}\left(\{[\bm{M}_0]_{k1},\ldots, [\bm{M}_0]_{kn}\}, [\bm{b}_0]_k, d, R\right)$}
  \label{Algo:DistRHop}
  \begin{algorithmic}
	\STATE \textbf{Part One:}	
	\STATE $\{[\bm{A}_0\bm{D}^{-1}_0]_{k1},\ldots,[\bm{A}_0\bm{D}^{-1}_0]_{kn}\} = \left\{\frac{[\bm{A}_0]_{k1}}{[\bm{D}_0]_{11}},\ldots,\frac{[\bm{A}_0]_{kn}}{[\bm{D}_0]_{nn}}\right\}$, $\{[\bm{D}^{-1}_0\bm{A}_0]_{k1},\ldots,[\bm{D}^{-1}_0\bm{A}_0]_{kn}\} = \{\frac{[\bm{A}_0]_{k1}}{[\bm{D}_0]_{kk}},\ldots,\frac{[\bm{A}_0]_{kn}}{[\bm{D}_0]_{kk}}\}$
	\STATE $[\bm{C}_0]_{k1},\ldots,[\bm{C}_0]_{kn} = \text{Comp}_0\left([\bm{M}_0]_{k1},\ldots, [\bm{M}_0]_{kn}, R\right)$, $[\bm{C}_1]_{k1},\ldots,[\bm{C}_1]_{kn}= \text{Comp}_1\left([\bm{M}_0]_{k1},\ldots, [\bm{M}_0]_{kn}, R\right)$
\\\hrulefill	
	\STATE \textbf{Part Two:}
	\FOR {$i=1$ to $d$}
		\STATE \textbf{if} \ {$i-1 < \rho$}
			\STATE $[\bm{u}^{(i-1)}_{1}]_k = [\bm{A}_0\bm{D}^{-1}_0\bm{b}_{i-1}]_k$
			\FOR {$j=2$ to $2^{i-1}$}
				\STATE $[\bm{u}^{(i-1)}_{j}]_k = [\bm{A}_0\bm{D}^{-1}_0\bm{u}^{(i-1)}_{j-1}]_k$
			\ENDFOR 
			\STATE $[\bm{b}_i]_k = [\bm{b}_{i-1}]_k + [\bm{u}^{(i-1)}_{2^{i-1}}]_k$
		\STATE \textbf{if} {$i-1 \ge \rho$}
			\STATE $l_{i-1} = \sfrac{2^{i-1}}{R}$
			\STATE $[\bm{u}^{(i-1)}_{1}]_k = [\bm{C}_0\bm{b}_{i-1}]_k$
			\FOR {$j=2$ to $l_{i-1}$}
				\STATE $[\bm{u}^{(i-1)}_{j}]_k = [\bm{C}_0\bm{u}^{(i-1)}_{j-1}]_k$
			\ENDFOR 
			\STATE $[\bm{b}_i]_k = [\bm{b}_{i-1}]_k + [\bm{u}^{(i-1)}_{l_{i-1}}]_k$
	\ENDFOR 
		\\\hrulefill
	\STATE \textbf{Part Three:}
	\STATE $[\bm{x}_d]_k = \sfrac{[\bm{b}_d]_k}{[\bm{D}_0]_{kk}}$
	\FOR {$i=d-1$ to $1$}
		\STATE \textbf{if} {$i < \rho$}
			\STATE $[\bm{\eta}^{(i+1)}_{1}]_k = [\bm{D}^{-1}_0\bm{A}_0\bm{x}_{i+1}]_k$
			\FOR {$j=2$ to $2^{i}$}
				\STATE $[\bm{\eta}^{(i+1)}_{j}]_k = [\bm{D}^{-1}_0\bm{A}_0\bm{\eta}^{(i+1)}_{j-1}]_k$
			\ENDFOR 
			\STATE $[\bm{x}_i]_k = \frac{1}{2}\left[\frac{[\bm{b}_i]_k}{[\bm{D}_{0}]_{kk}} + [\bm{x}_{i+1}]_k + [\bm{\eta}^{i + 1}_{2^i}]_{k}\right]$
		\STATE \textbf{if} {$i \ge \rho$}
			\STATE $l_i = \sfrac{2^i}{R}$
			\STATE $[\bm{\eta}^{(i+1)}_{1}]_k = [\bm{C}_1\bm{x}_{i+1}]_k$
			\FOR {$j=2$ to $l_i$}
				\STATE $[\bm{\eta}^{(i+1)}_{j}]_k = [\bm{C}_1\bm{\eta}^{(i+1)}_{j-1}]_k$
			\ENDFOR 
			\STATE $[\bm{x}_i]_k = \frac{1}{2}\left[\frac{[\bm{b}_i]_k}{[\bm{D}_0]_{kk}}  + [\bm{x}_{i+1}]_k + [\bm{\eta}^{i+1}_{l_i}]_k \right]$
	\ENDFOR 
	\STATE $[\bm{x}_0]_k = \frac{1}{2}\left[\frac{[\bm{b}_0]_k}{[\bm{D}_0]_{kk}} + [\bm{x}_1]_k + [\bm{D}^{-1}_0\bm{A}_0\bm{x}_1]_k \right]$
    \STATE \textbf{return} $[\bm{x}_0]_k$ 
  \end{algorithmic}
\end{algorithm}
\begin{algorithm}[h!]
  \caption{$\text{Comp}_0\left([\bm{M}_0]_{k1},\ldots, [\bm{M}_0]_{kn}, R\right)$}
  \label{Alg_4}
  \begin{algorithmic}
	\FOR {$l = 1$ to $R - 1$}
		\FOR {$j$ s.t.$\bm{v}_j\in \mathbb{N}_{l+1}(\bm{v}_k)$}
			\STATE
			$\left[(\bm{A}_0\bm{D}^{-1}_0)^{l+1}\right]_{kj} = \sum\limits_{r:\bm{v}_r\in \mathbb{N}_1(v_j)}\frac{[\bm{D}_0]_{rr}}{[\bm{D}_0]_{jj}}[(\bm{A}_0\bm{D}^{-1}_0)^l]_{kr}[\bm{A}_0\bm{D}^{-1}_0]_{jr}$
		\ENDFOR
	\ENDFOR
	\STATE \textbf{return} $\bm{c}_0 = \{[(\bm{A}_0\bm{D}^{-1}_0)^{R}]_{k1},\ldots,[(\bm{A}_0\bm{D}^{-1}_0)^{R}]_{kn} \}$
  \end{algorithmic}
\end{algorithm}
\begin{algorithm}
  \caption{$\text{Comp}_1([\bm{M}_0]_{k1},\ldots, [\bm{M}_0]_{kn}, R)$}
  \label{Alg_5}
  \begin{algorithmic}
	\FOR {$l = 1$ to $R - 1$}
		\FOR {$j$ s.t.$\bm{v}_j\in \mathbb{N}_{l+1}(\bm{v}_k)$}
			\STATE$
			\left[(\bm{D}^{-1}_0\bm{A}_0)^{l+1}\right]_{kj} = \sum\limits_{r:\bm{v}_r\in \mathbb{N}_1(\bm{v}_j)}\frac{[\bm{D}_0]_{jj}}{[\bm{D}_0]_{rr}}[(\bm{D}^{-1}_0\bm{A}_0)^l]_{kr}[\bm{D}^{-1}_0\bm{A}_0]_{jr}$
		\ENDFOR
	\ENDFOR
	\STATE \textbf{return} $\bm{c}_1 = \{[(\bm{D}^{-1}_0\bm{A}_0)^{R}]_{k1},\ldots,[(\bm{D}^{-1}_0\bm{A}_0)^{R}]_{kn} \}$
  \end{algorithmic}
\end{algorithm}

\textbf{Analysis of Algorithm~\ref{Algo:DistRHop}} The following Lemma shows that $\text{RDistRSolve}$ computes the $k^{th}$ component of the ``crude'' approximation of $\bm{x}^{\star}$ and provides the algorithm's time complexity 
 
\begin{lemma}\label{r_hop_rude_lemma}
Let $\bm{M}_0 = \bm{D}_0 - \bm{A}_0$ be the standard splitting and let $\bm{Z}^{\prime}_0$ be the operator defined by $\text{RDistRSolve}$, namely, $\bm{x}_0 = \bm{Z}^{\prime}_0\bm{b}_0$. Then,
$\bm{Z}^{\prime}_0\approx_{\epsilon_d} \bm{M}^{-1}_0$.  
$\text{RDistRSolve}$ requires $\mathcal{O}\left(\frac{2^d}{R}\beta + \beta Rd_{max}\right)$, where $\beta = \min\left\{n, \frac{\left(d^{R+1}_{\text{max}} - 1\right)}{\left(d_{\text{max}} - 1\right)}\right\}$, to arrive at $\bm{x}_{0}$. 
\end{lemma}
\begin{proof}
See Appendix.
\end{proof}

\subsection{``Exact'' Distributed R-Hop SDDM Solver}
Next, we provide the exact R-Hop solver. Similar to $\text{RDistRSolve}$, each node $\bm{v}_k$ receives the $k^{th}$ row $\bm{M}_0$, $[\bm{b}_0]_k$, $d$, $R$, and a precision parameter $\epsilon$ as inputs, and outputs the $k^{th}$ component of the $\epsilon$ close approximation of vector $\bm{x}^{\star}$. 

\begin{algorithm}
  \caption{ \hspace{-.3em} $\text{EDistRSolve}\left(\{[\bm{M}_0]_{k1},\ldots, [\bm{M}_0]_{kn}\}, [\bm{b}_0]_k, d, R,\epsilon\right)$}
\label{Algo:EDistR}
  \begin{algorithmic}
\STATE \textbf{Initialize}: $[\bm{y}_0]_k = 0$, and $[\bm{\chi}]_k = \text{RDistRSolve}(\{[M_0]_{k1},\ldots, [M_0]_{kn}\}, [b_0]_k, d, R)$	
	\FOR {$t=1$ to $q$}
		\STATE  \hspace{-2em} $[\bm{u}^{(1)}_{t}]_k = [\bm{D}_0]_{kk}[\bm{y}_{t-1}]_k - \sum_{j: \bm{v}_j\in \mathbb{N}_{1}(\bm{v}_k)}[\bm{A}_{0}]_{kj}[\bm{y}_{t-1}]_j$
		 \STATE \hspace{-2em} $[\bm{u}^{(2)}_{t}]_k = \text{RDistRSolve}(\{[\bm{M}_0]_{k1},\ldots, [\bm{M}_0]_{kn}\}, [\bm{u}^{(1)}_{t}]_k, d, R)$
		\STATE  \hspace{-2em} $[\bm{y}_t]_k = [\bm{y}_{t-1}]_k - [\bm{u}^{(2)}_{t}]_k + [\bm{\chi}]_k$ 
	\ENDFOR \textbf{end for}	
    \STATE \textbf{return} $[\tilde{\bm{x}}]_k = [\bm{y}_q]_k$
  \end{algorithmic}
\end{algorithm}

\textbf{Analysis of Algorithm~\ref{Algo:EDistR}:} The following Lemma shows that $\text{EDistRSolve}$ computes the $k^{th}$ component of the $\epsilon$ close approximation to $\bm{x}^{\star}$ and provides the time complexity analysis. 

\begin{lemma}\label{Dist_Exact_algorithm_guarantee_lemma}
Let $\bm{M}_0 = \bm{D}_0 - \bm{A}_0$ be the standard splitting. Further, let $\epsilon_d < \sfrac{1}{3}\ln2$. Then Algorithm~\ref{Algo:EDistR} requires $\mathcal{O}\left(\log\frac{1}{\epsilon}\right)$ iterations to return the $k^{th}$ component of the $\epsilon$ close approximation to $\bm{x}^{\star}$.  
\end{lemma}
\begin{proof}
See Appendix.
\end{proof} 
Next, the following Lemma provides the time complexity analysis of $\text{EDistRSolve}$. 

\begin{lemma}\label{time_complexity_of_distresolve}
Let $\bm{M}_0 = \bm{D}_0 - \bm{A}_0$ be the standard splitting and let $\epsilon_d < \sfrac{1}{3}\ln 2 $, then $\text{EDistRSolve}$ requires 
$\mathcal{O}\left(\left(\sfrac{2^d}{R}\beta + \beta Rd_{max}\right)\log\left(\sfrac{1}{\epsilon}\right)\right)$ time steps. Moreover, for each node $\bm{v}_k$, $\text{EDistRSolve}$ only uses information from the R-hop neighbors.
\end{lemma}

\begin{proof}
See Appendix.
\end{proof}
The complexity of the proposed algorithms depend on the length of the inverse approximated chain, $d$. Here, we provide an analysis to determine the value of $d$ which guarantees $\epsilon_d < \frac{1}{3}\ln2$ in $\mathcal{C} = \{\bm{A}_0,\bm{D}_0,\bm{A}_1, \bm{D}_1,\ldots, \bm{A}_d, \bm{D}_d\}$. These results are summarized the following lemma

\begin{lemma}\label{eps_d_lemma}
Let $\bm{M}_0 = \bm{D}_0 - \bm{A}_0$ be the standard splitting and let $\kappa$ denote the condition number of $\bm{M}_0$. Consider the inverse approximated chain $\mathcal{C} = \{\bm{A}_0,\bm{D}_0,\bm{A}_1, \bm{D}_1,\ldots, \bm{A}_d, \bm{D}_d\}$ with length $d =  \lceil \log \left(2\ln\left(\frac{\sqrt[3]{2}}{\sqrt[3]{2} - 1}\right)\kappa\right)\rceil$, then 
$\bm{D}_0\approx_{\epsilon_d} \bm{D}_0 - \bm{D}_0\left(\bm{D}^{-1}_0\bm{A}_0\right)^{2^d}$, 
with $\epsilon_d < \sfrac{1}{3}\ln2$.
\end{lemma}

\begin{proof}
See Appendix. 
\end{proof}

Combining the above results finalizes the proof of Lemma~\ref{lemma:SDDM}. The usage of this distributed solver to approximate the Newton direction, as detailed in the next section, enables fast and accurate distributed Newton methods capable of approximating centralized Newton directions up to any arbitrary $\epsilon$. 

\section{FAST \& ACCURATE DISTRIBUTED NEWTON METHOD}\label{Sec:NewtonApprox}
Our approach only requires R-Hop communication for the distributed approximation of the Newton direction. Given the results of Lemma~\ref{lemma:Crap}, we can determine the approximate Newton direction by solving a system of linear equations represented by an SDD matrix\footnote{Due to space constraints, we refrain some of the proofs to the appendix.} according to Section~\ref{Sec:SDDM}, with $\bm{M}_{0}=\bm{H}_{k}=\bm{H}(\bm{\lambda}_{k})$.  

Formally, we consider the following iteration scheme:
\begin{equation}
\bm{\lambda}_{k+1}=\bm{\lambda}_{k} + \alpha_{k}\tilde{\bm{d}}_{k},
\end{equation}
with $k$ representing the iteration number, $\alpha_{k}$ the step-size, and $\tilde{\bm{d}}_{k}$ denoting the approximate Newton direction. We determine $\tilde{\bm{d}}_{k}$ by solving $\bm{H}_{k}\bm{d}_{k}=-\bm{g}_{k}$ using Algorithm~\ref{Algo:EDistR}. It is easy to see that our approximation of the Newton direction, $\tilde{\bm{d}}_{k}$, satisfies
\begin{align*}
||\tilde{\bm{d}}_{k} - \bm{d}_{k}||_{\bm{H}_{k}} &\leq \epsilon ||\bm{d}_{k}||_{\bm{H}_{k}} \\
\text{with} \ \ \ \tilde{\bm{d}}_{k} &=-\bm{Z}_{k}\bm{g}_{k},
\end{align*}
where $\bm{Z}_{k}$ approximates $\bm{H}^{\dagger}_{k}$ according to the routine of Algorithm~\ref{Algo:EDistR}. The accuracy of this approximation is quantified in the following Lemma 

\begin{lemma}\label{Lemma:Bla}
Let $\bm{H}_{k}=\bm{H}(\bm{\lambda}_{k})$ be the Hessian of the dual function, then for any arbitrary $\epsilon > 0$ we have  
\begin{equation*}
e^{-\epsilon^{2}} \bm{v}^{\mathsf{T}} \bm{H}_{k}^{\dagger} \bm{v} \leq \bm{v}^{\mathsf{T}}\bm{Z}_{k}\bm{v}\leq e^{\epsilon^{2}}\bm{v}^{\mathsf{T}}\bm{H}_{k}^{\dagger}\bm{v}, \ \ \ \ \ \ \forall \bm{v} \in \bm{1}^{\perp}.
\end{equation*}
\end{lemma}
\begin{proof}
See Appendix. 
\end{proof}

Given such an accurate approximation, next we analyze the iteration scheme of our proposed method showing that similar to standard Newton methods, we achieve superlinear convergence within a neighborhood of the optimal value. We start by analyzing the change in the Laplacian norm of the gradient between two successive iterations 
\begin{lemma}
Consider the following iteration scheme $\bm{\lambda}_{k+1}=\bm{\lambda}_{k} + \alpha_{k}\tilde{\bm{d}}_{k}$ with $\alpha_{k} \in (0,1]$, then, for any arbitrary $\epsilon>0$, the Laplacian norm of the gradient, $||\bm{g}_{k+1}||_{\mathcal{L}}$, follows:
\begin{align}\label{Eq:NormG}
||\bm{g}_{k+1}||_{\mathcal{L}} &\leq \left[1-\alpha_{k} +\alpha_{k}\epsilon\frac{\mu_{n}(\mathcal{L})}{\mu_{2}(\mathcal{L})}\sqrt{\frac{\Gamma}{\gamma}}\right]||\bm{g}_{k}||_{\mathcal{L}}  \\ \nonumber 
& \hspace{10em}+ \frac{\alpha_{k}^{2}B\Gamma^{2}(1+\epsilon)^{2}}{2\mu^{2}_{2}(\mathcal{L})}||\bm{g}_{k}||_{\mathcal{L}}^{2},
\end{align}
with $\mu_{n}(\mathcal{L})$ and $\mu_{2}(\mathcal{L})$ being the largest and second smallest eigenvalues of $\mathcal{L}$, $\Gamma$ and $\gamma$ denoting the upper and lower bounds on the dual's Hessian, and $B \in \mathbb{R}$ is defined in Lemma~\ref{lemma:B}. 
\end{lemma}
\begin{proof}
See Appendix. 
\end{proof}
At this stage, we are ready to present the main results quantifying the convergence phases exhibited by our approach:
\begin{theorem}
Let $\gamma$, $\Gamma$, $B$ be the constants defined in Assumption~\ref{Ass:Two} and Lemma~\ref{lemma:Crap}, $\mu_{n}(\mathcal{L})$ and $\mu_{2}(\mathcal{L})$ representing the largest and second smallest eigenvalues of the normalized laplacian $\mathcal{L}$, $\epsilon \in \left(0, \frac{\mu_{2}(\mathcal{L}}{\mu_{n}(\mathcal{L})}\sqrt{\frac{\Gamma}{\gamma}}\right)$ the precision parameter for the SDDM (Section~\ref{Sec:SDDM}) solver, and letting the optimal step-size parameter $\alpha^{*}=\frac{e^{-\epsilon^{2}}}{(1+\epsilon)^{2}}\left(\frac{\gamma}{\Gamma}\frac{\mu_{2}(\mathcal{L})}{\mu_{n}(\mathcal{L})}\right)^{2}$. Then the proposed algorithm given by the $\bm{\lambda}_{k+1}=\bm{\lambda}_{k}+\alpha^{*}\tilde{\bm{d}}_{k}$ exhibits the following three phases of convergence:
\begin{enumerate}
\item \textbf{Strict Decreases Phase:} While $||\bm{g}_{k}||_{\mathcal{L}} \geq \eta_{1}$:
\begin{equation*}
q(\bm{\lambda}_{k+1})-q(\bm{\lambda}_{k}) \leq -\frac{1}{2} \frac{e^{-2\epsilon^{2}}}{(1+\epsilon)^{2}}\frac{\gamma^{3}}{\Gamma^{2}}\frac{\mu_{2}^{2}(\mathcal{L})}{\mu_{n}^{4}(\mathcal{L})}\eta_{1}^{2}.
\end{equation*} 
\item \textbf{Quadratic Decrease Phase:} While $\eta_{0} \leq ||\bm{g}_{k}||_{\mathcal{L}}\eta_{1}$:
\begin{equation*}
||\bm{g}_{k+1}||_{\mathcal{L}} \leq \frac{1}{\eta_{1}}||\bm{g}_{k}||_{\mathcal{L}}^{2}.
\end{equation*}
\item \textbf{Terminal Phase:} When $||\bm{g}_{k}||_{\mathcal{L}} \leq \eta_{0}$: 
\begin{equation*}
||\bm{g}_{k+1}||_{\mathcal{L}}\leq \sqrt{\left[1-\alpha^{*}+\alpha^{*}\epsilon\frac{\mu_{n}(\mathcal{L})}{\mu_{2}(\mathcal{L})}\sqrt{\frac{\Gamma}{\gamma}}\right]}||\bm{g}_{k}||_{\mathcal{L}},
\end{equation*}
\end{enumerate}
where $\eta_{0}=\frac{\bm{\xi}(1-\bm{\xi})}{\bm{\zeta}}$ and $\eta_{1}=\frac{1-\bm{\xi}}{\bm{\zeta}}$, with 
\begin{align}\label{Eq:Aux}
\bm{\xi}&=\sqrt{\left[1-\alpha^*+\alpha^*\epsilon\frac{\mu_{n}(\mathcal{L})}{\mu_{2}(\mathcal{L})}\sqrt{\frac{\Gamma}{\gamma}}\right]}\\\nonumber
\bm{\zeta}&=\frac{B(\alpha^{*}\Gamma(1+\epsilon))^{2}}{2\mu_{2}^{2}(\mathcal{L})}
\end{align}
\end{theorem}

\begin{proof}
We will proof the above theorem by handling each of the cases separately. We start by considering the case when $||\bm{g}_{k}||_{\mathcal{L}} > \eta_{1}$ (i.e., \textbf{Strict Decrease Phase}). We have:
\begin{align*}
q(\lambda_{k+1}) & = q(\lambda_{k}) + \bm{g}_{k}^{\mathsf{T}}(\bm{\lambda}_{k+1}-\bm{\lambda}_{k})  \\
&\hspace{5em}+\frac{1}{2}(\bm{\lambda}_{k+1}-\bm{\lambda}_{k})^{\mathsf{T}}\bm{H}(\bm{z})(\lambda_{k+1}-\bm{\lambda}_{k}) \\ 
& = q(\bm{\lambda}_{k}) + \alpha_{k} \bm{g}_{k}^{\mathsf{T}}\tilde{\bm{d}}_{k} +\frac{\bm{\alpha}_{k}^{2}}{2} \tilde{\bm{d}}_{k}^{\mathsf{T}}\bm{H}(\bm{z})\tilde{\bm{d}}_{k} \\
& \leq q(\bm{\lambda}_{k}) + \alpha_{k}\bm{g}_{k}^{\mathsf{T}}\tilde{\bm{d}}_{k}+\frac{\alpha_{k}^{2}}{2\gamma}\tilde{\bm{d}}^{\mathsf{T}}_{k} \mathcal{L}\tilde{\bm{d}}_{k},
\end{align*} 
where the last steps holds since $\bm{H}(\cdot) \preceq \frac{1}{\gamma}\mathcal{L}$. Noticing that $||\tilde{\bm{d}}_{k}||_{\mathcal{L}}^{2} \leq \frac{\Gamma^{2}(1+\epsilon)^{2}}{\mu_{2}^{2}(\mathcal{L})}||\bm{g}_{k}||_{\mathcal{L}}^{2}$ (see Appendix), the only remaining step needed is to evaluate $\bm{g}_{k}^{\mathsf{T}}\tilde{\bm{d}}_{k}$. Knowing that $\tilde{\bm{d}}_{k} = -\bm{Z}_{k}\bm{g}_{k}$, we recognize
\begin{align*}
\bm{g}_{k}^{\mathsf{T}}\tilde{\bm{d}}_{k} &= - \bm{g}_{k}^{\mathsf{T}}\bm{Z}_{k}\bm{g}_{k} \leq e^{-\epsilon^{2}}\bm{g}_{k}^{\mathsf{T}}\bm{H}_{k}^{\dagger}\bm{g}_{k} \ (\text{Lemma~\ref{Lemma:Bla}}) \\
&\leq -\frac{e^{-\epsilon^{2}}}{\mu_{n}(\bm{H}_{k})}\bm{g}_{k}^{\mathsf{T}}\bm{g}_{k} \leq -\frac{e^{-\epsilon}}{\mu_{n}(\mathcal{L})}\bm{g}_{k}^{\mathsf{T}}\bm{g}_{k} \\
&\leq -\frac{e^{-\epsilon^{2}}\gamma}{\mu_{n}(\mathcal{L})}\frac{\bm{g}_{k}^{\mathsf{T}}\mathcal{L}\bm{g}_{k}}{\mu_{n}(\mathcal{L})} = \frac{e^{-\epsilon^{2}}\gamma}{\mu_{n}^{2}(\mathcal{L})}||\bm{g}_{k}||_{\mathcal{L}}^{2},
\end{align*}
where the last step follows from the fact that $\forall \bm{v}\in \mathbb{R}^{n}: \bm{v}^{\mathsf{T}}\bm{v} \geq \frac{\bm{v}^{\mathsf{T}}\mathcal{L}\bm{v}}{\mu_{n}(\mathcal{L})}$. Therefore, we can write
\begin{equation*}
q(\bm{\lambda}_{k+1})-q(\bm{\lambda}_{k}) \leq -\left[\alpha_{k}\frac{e^{-\epsilon^{2}}\gamma}{\mu_{n}^{2}(\mathcal{L})}-\alpha_{k}^{2}\frac{\Gamma^{2}(1+\epsilon)^{2}}{2\gamma\mu_{2}^{2}(\mathcal{L})}\right]||\bm{g}_{k}||_{\mathcal{L}}^{2}.
\end{equation*}
It is easy to see that $\alpha_{k}=\alpha^{*}=\frac{e^{-\epsilon^{2}}}{(1+\epsilon)^{2}}\left(\frac{\gamma}{\Gamma}\frac{\mu_{2}(\mathcal{L})}{\mu_{n}(\mathcal{L})}\right)^{2}$ minimizes the right-hand-side of the above equation. Using $||\bm{g}_{k}||_{\mathcal{L}}$ gives the constant decrement in the dual function between two successive iterations as 
\begin{equation*}
q(\bm{\lambda}_{k+1})-q(\bm{\lambda}_{k}) \leq -\frac{1}{2} \frac{e^{-2\epsilon^{2}}}{(1+\epsilon)^{2}}\frac{\gamma^{3}}{\Gamma^{2}}\frac{\mu_{2}^{2}(\mathcal{L})}{\mu_{n}^{4}(\mathcal{L})}\eta_{1}^{2}.
\end{equation*}
Considering the case when $\eta_{0} \leq ||\bm{g}_{k}||_{\mathcal{L}}^{2}\eta_{1}$ (i.e., \textbf{Quadratic Decrease Phase}), Equation~\ref{Eq:NormG} can be rewritten as
\begin{equation*}
||\bm{g}_{k+1}||_{\mathcal{L}}\leq \bm{\xi}^{2}||\bm{g}_{k}||_{\mathcal{L}} + \bm{\zeta}||\bm{g}_{k}||_{\mathcal{L}}^{2},
\end{equation*}
with $\bm{\xi}$ and $\bm{\zeta}$ defined as in Equation~\ref{Eq:Aux}. Further, noticing that since $||\bm{g}_{k}||_{\mathcal{L}} \geq \eta_{0}$ then $||\bm{g}_{k}||_{\mathcal{L}}\leq \frac{1}{\eta_{0}}||\bm{g}_{k}||_{\mathcal{L}}^{2}=\frac{\bm{\zeta}}{\bm{\xi}(1-\bm{\xi})}||\bm{g}_{k}||_{\mathcal{L}}^{2}$. Consequently the quadratic decrease phase is finalized by
\begin{align*}
||\bm{g}_{k+1}||_{\mathcal{L}} \leq \bm{\zeta}\left(\frac{\bm{\xi}}{1-\bm{\xi}}+1\right)||\bm{g}_{k}||_{\mathcal{L}}^{2}&=\frac{\bm{\zeta}}{1-\bm{\xi}}||\bm{g}_{k}||_{\mathcal{L}}^{2} \\ 
&=\frac{1}{\eta_{1}}||\bm{g}_{k}||_{\mathcal{L}}^{2}.
\end{align*}
Finally, we handle the case where $||\bm{g}_{k}||_{\mathcal{L}}\leq \eta_0$ (i.e., \textbf{Terminal Phase}). Since $||\bm{g}_{k}||_{\mathcal{L}}^{2} \leq \eta_{0}||\bm{g}_{k}||_{\mathcal{L}}$, it is easy to see that 
\begin{align*}
||\bm{g}_{k+1}||_{\mathcal{L}} &\leq (\bm{\xi}^{2}+\bm{\zeta}\eta_{0})||\bm{g}_{k}||_{\mathcal{L}}=(\bm{\xi}^{2}+\bm{\xi}(1-\bm{\xi})) ||\bm{g}_{k}||_{\mathcal{L}} \\
& \bm{\xi}||\bm{g}_{k}||_{\mathcal{L}} = \sqrt{\left[1-\alpha^*+\alpha^*\epsilon\frac{\mu_{n}(\mathcal{L})}{\mu_{2}(\mathcal{L})}\sqrt{\frac{\Gamma}{\gamma}}\right]}||\bm{g}_{k}||_{\mathcal{L}}.
\end{align*}
\end{proof}

Having proved the three convergence phases of our algorithm, we next analyze the number of iterations needed by each phase. These results are summarized in the following lemma: 

\begin{lemma}
Consider the algorithm given by the following iteration protocol: $\bm{\lambda}_{k+1}=\bm{\lambda}_{k+1}+\alpha^{*}\tilde{\bm{d}}_{k}$. Let $\bm{\lambda}_{0}$ be the initial value of the dual variable, and $q^{*}$ be the optimal value of the dual function. Then, the number of iterations needed by each of the three phases satisfy: 
\begin{enumerate}
\item The {\textbf{strict decrease phase}} requires the following number iterations to achieve the quadratic phase: 
\begin{equation*}
N_{1} \leq C_{1} \frac{\mu_{n}(\mathcal{L})^{2}}{\mu_{2}^{3}(\mathcal{L})} \left[1-\epsilon\frac{\mu_{n}(\mathcal{L})}{\mu_{2}(\mathcal{L})}\sqrt{\frac{\Gamma}{\gamma}}\right]^{-2},
\end{equation*}
where $C_{1}=C_{1}\left(\epsilon, \gamma, \Gamma, \bm{\delta}, q(\bm{\lambda}_{0}), q^{\star} \right)=2\bm{\delta}^{2}(1+\epsilon)^{2}\left[q(\bm{\lambda}_{0})-q^{\star}\right]\frac{\Gamma^{2}}{\gamma}$.
\item The \textbf{quadratic decrease phase} requires the following number of iterations to terminate: 
\begin{equation*}
N_{2} = \log_{2}\left[\frac{\frac{1}{2}\log_{2}\left(\left[1-\alpha^{*}\left(1-\epsilon\frac{\mu_{n}(\mathcal{L})}{\mu_{2}(\mathcal{L})}\sqrt{\frac{\Gamma}{\gamma}}\right)\right]\right)}{\log_{2}(r)}\right],
\end{equation*}
where $r=\frac{1}{\eta_{1}}||\bm{g}_{k^{\prime}}||_{\mathcal{L}}$, with $k^{\prime}$ being the first iteration of the quadratic decrease phase. 
\item The radius of the {\textbf{terminal phase}} is characterized by:
\begin{equation*}
\rho_{\text{terminal}} \leq \frac{2\left[1-\epsilon\frac{\mu_{n}(\mathcal{L})}{\mu_{2}(\mathcal{L})}\sqrt{\frac{\Gamma}{\gamma}}\right]}{e^{-\epsilon^{2}\gamma\bm{\delta}}}\mu_{n}(\mathcal{L})\sqrt{\mu_{2}(\mathcal{L})}.
\end{equation*}
\end{enumerate}
\end{lemma}
\begin{proof}
See Appendix. 
\end{proof}
Given the above result, the total message complexity can then be derived as $\mathcal{O}\left(\left(N_{1}+N_{2}\right)n\beta\left(\kappa(\bm{H}_{k})\frac{1}{R}+R\text{d}_{\max}\right)\log\left(\frac{1}{\epsilon}\right)\right)$. 
\section{EXPERIMENTS AND RESULTS}\label{Sec:Exp}

\begin{figure*}[t!]
\centering
\vspace{-.5em}
\subfigure[$||\bm{A}\bm{x}-\bm{b}||$ on a random network with 30 nodes and 70 edges]{
	\label{fig:PerfSM}
\includegraphics[width=0.23\textwidth,height=1.3in,clip,trim=0.0in 0.00in 0.65in 0.2in]{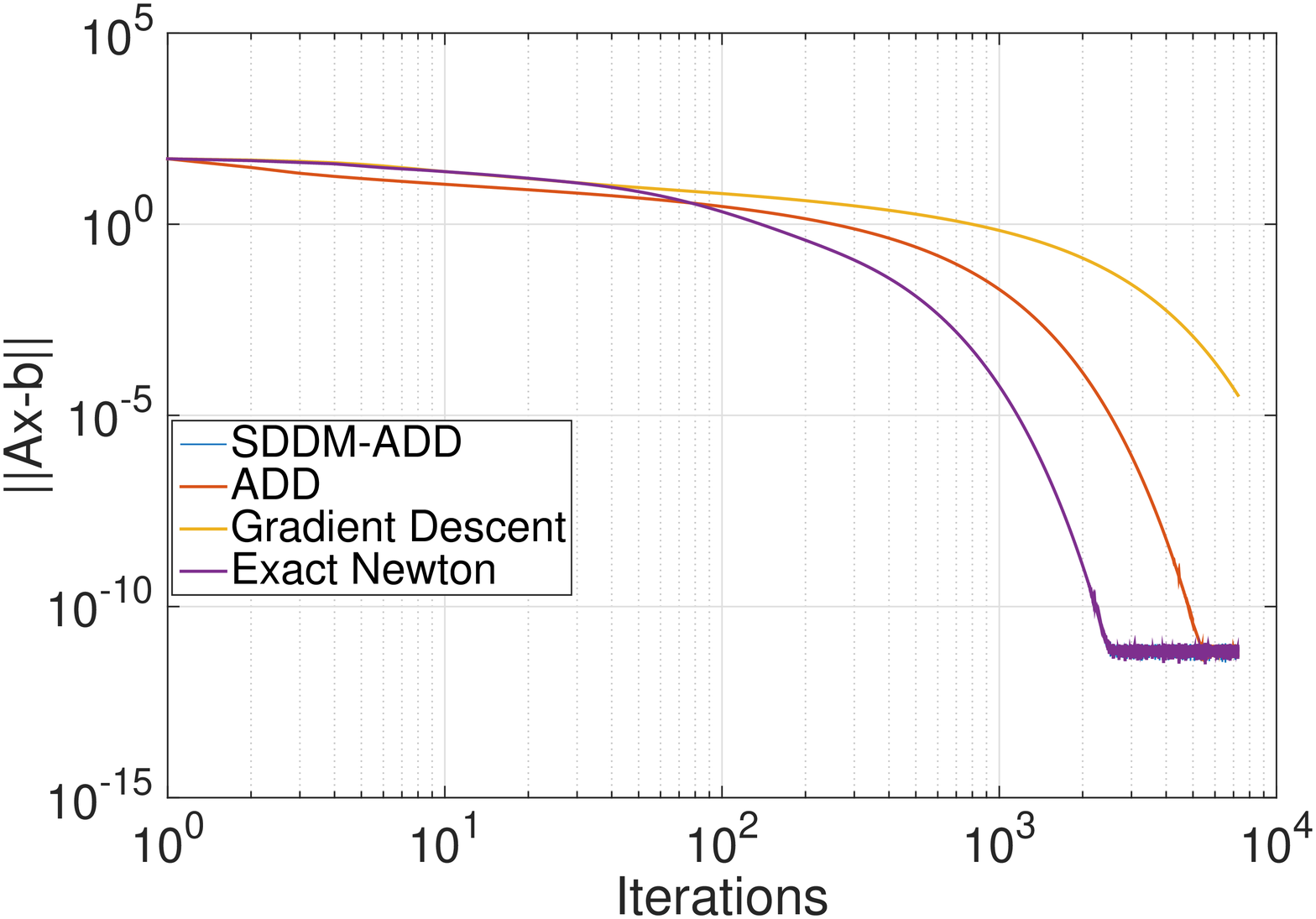}
}
\hfill\hspace{-.3em}\hfill
\subfigure[$f\left(\bm{x}_{k}\right)$ on a random network with 30 nodes and 70 edges]{
	\label{fig:PerfCP}
\includegraphics[width=0.23\textwidth,height=1.34in,clip,trim=0.0in 0.0in 0.0in 0.0in]{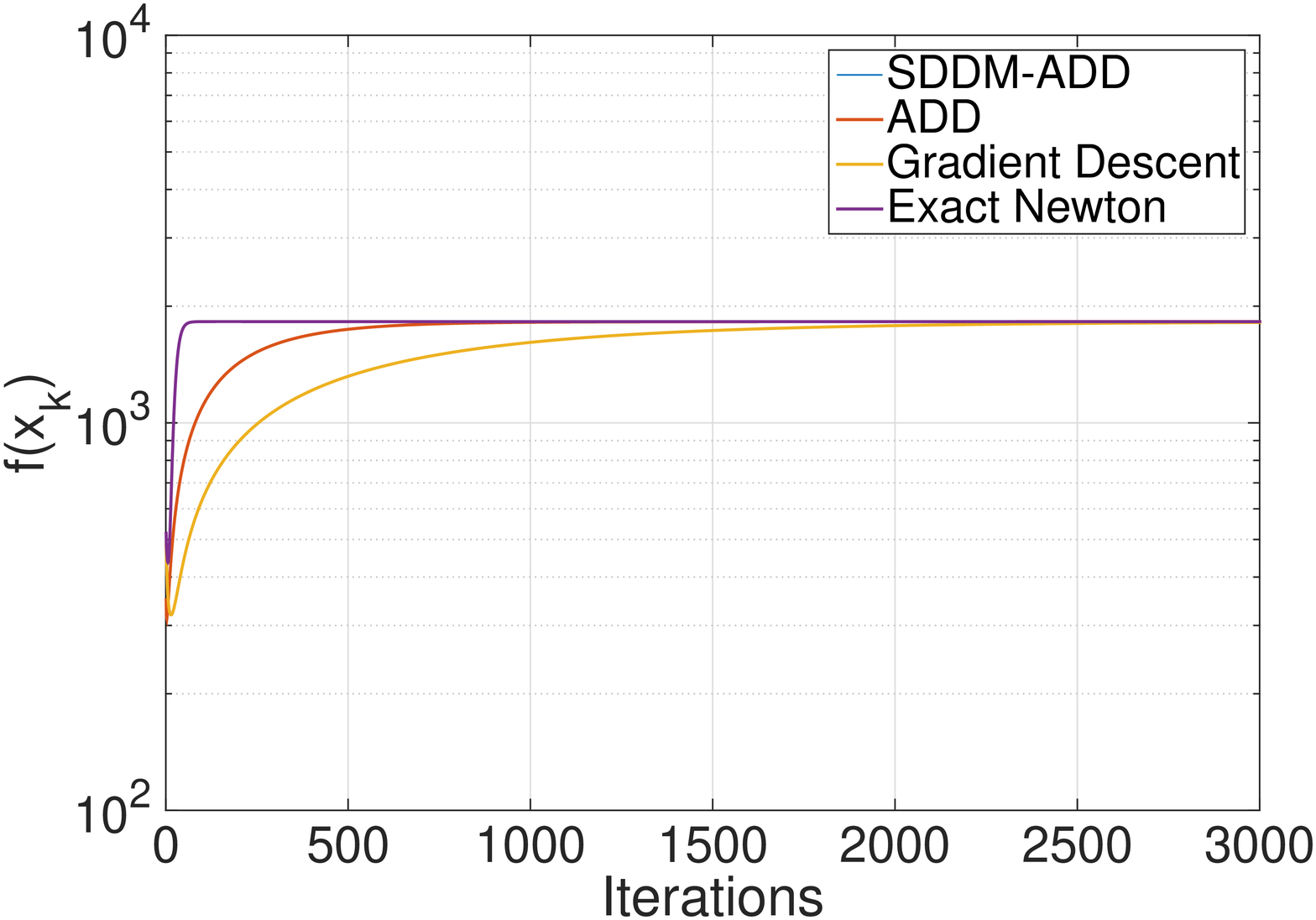}
}
\hfill\hspace{-.5em}
\subfigure[$||\bm{A}\bm{x}_{k}-\bm{b}||$ on a random network with 90 nodes and 200 edges]{
	\label{fig:TrajSM}
\includegraphics[width=0.23\textwidth,height=1.3in,clip,trim=0.15in 0.0in 0.0in 0.12in]{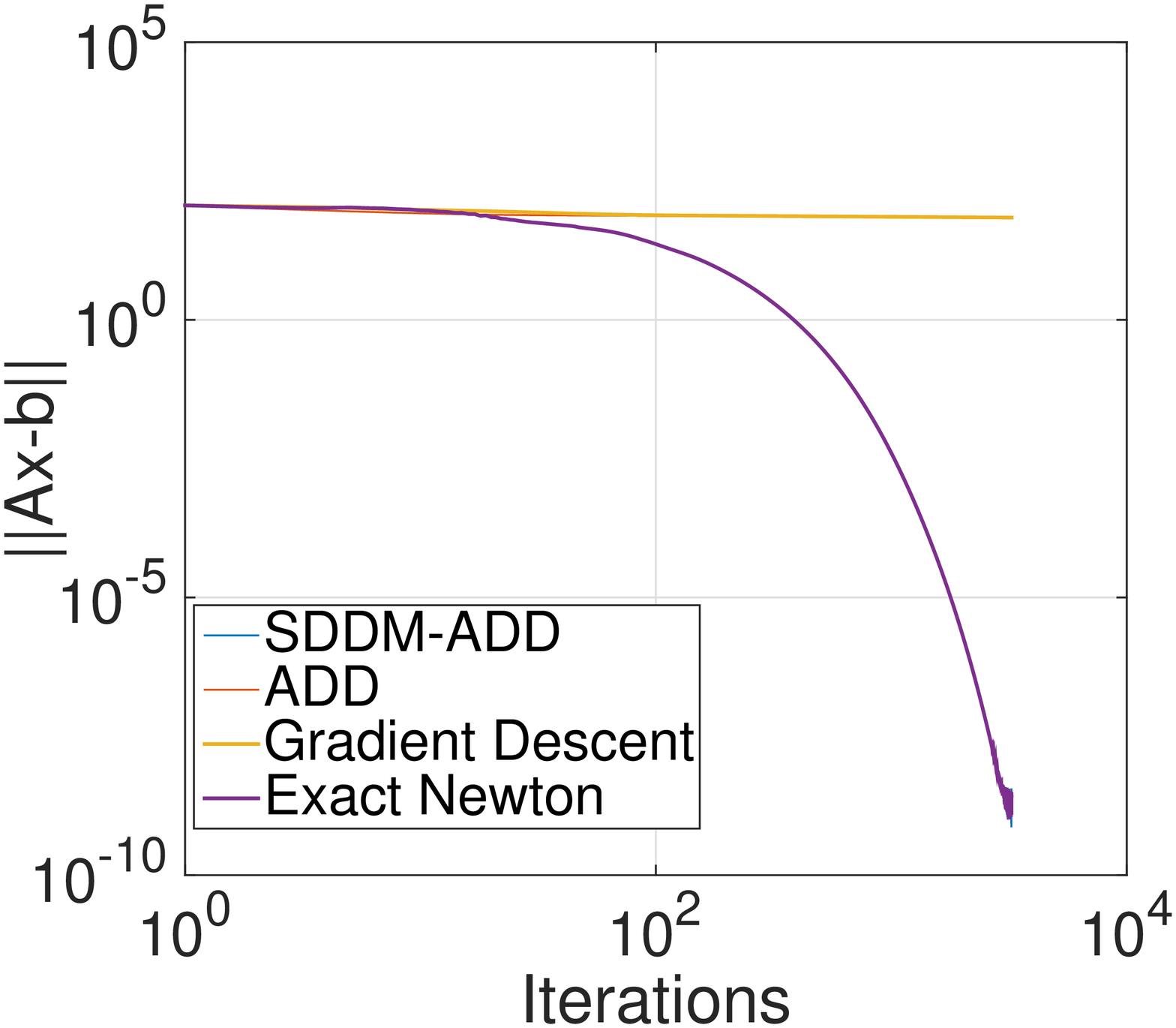}
}
\hfill
\hfill\hspace{-.5em}\hfill
\subfigure[$f(\bm{x}_{k})$ on a random network with 90 nodes and 200 edges]{
	\label{fig:TrajCP}
\includegraphics[width=0.24\textwidth,height=1.3in,clip,trim=0.15in 0.0in 0.0in 0.1in]{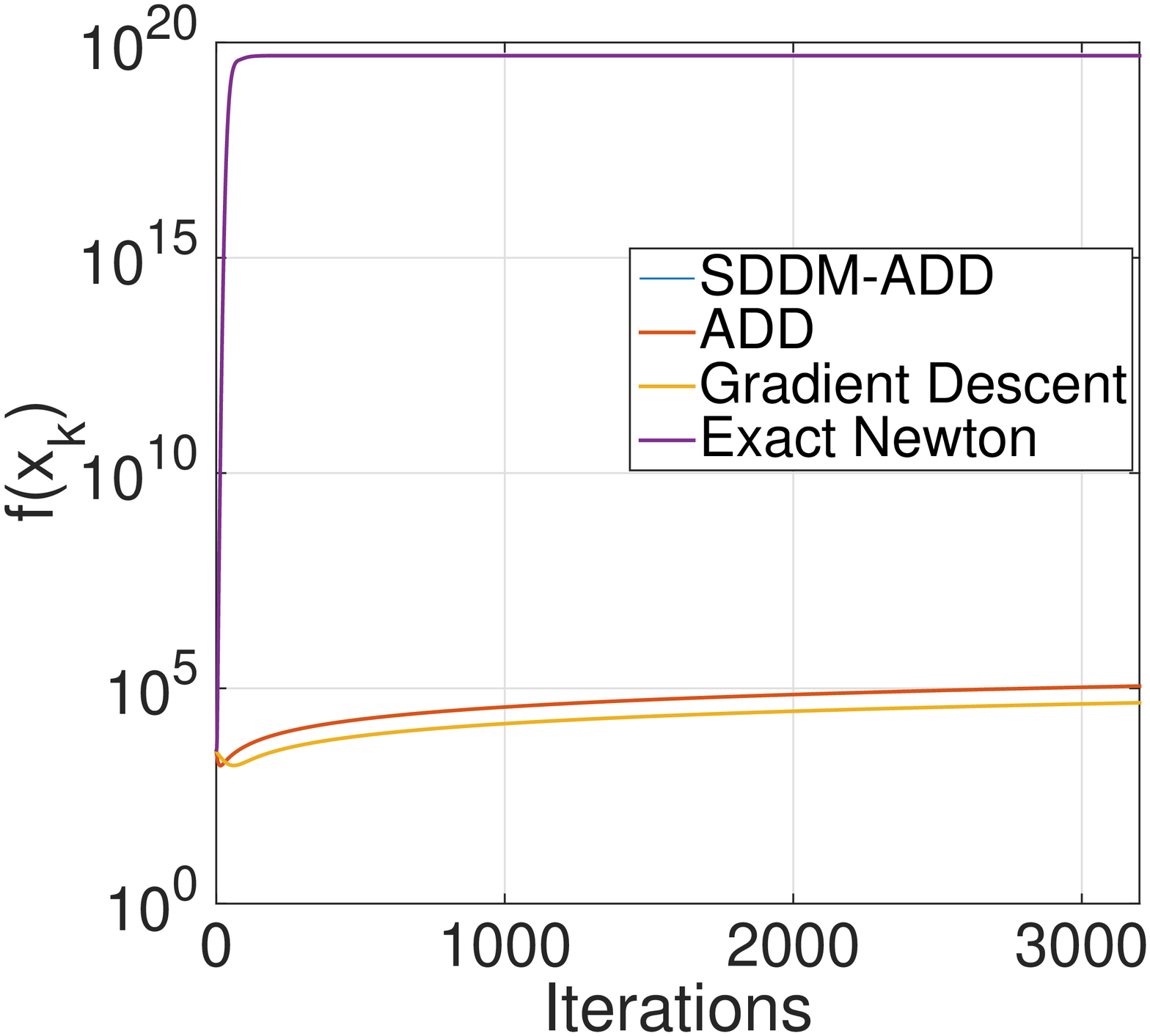}
}
\hfill
\caption{Performance metrics on two randomly generated networks, showing the primal objective, $f\left(\bm{x}_{k}\right)$, and feasibility $||\bm{A}\bm{x}_{k}-\bm{b}||$ as a function of the number of iterations $k$. On a relatively small network (i.e., $30$ nodes and $70$ edges) we outperform ADD and gradient descent by approximately an order of magnitude. On larger networks (i.e., 90 nodes and 200 edges), SDDM-ADD is superior to both ADD and gradient descent, where the primal objective of the latter two algorithms converges to $10^{5}$ after 3000 iterations. It is also worth noting that we perform closely to the exact Newton method computed according to a centralized approach.}
\label{Fig:ResBenchmark}
\vspace{-.6em}
\end{figure*}
We evaluated our approach on two randomly generated networks. The first consisted of 30 nodes and 70 edges, while the second contained 90 nodes with 200 edges. The edges were chosen uniformly at random. The flow vectors, $\bm{b}$, were chosen to place source and sink nodes $\text{diam}(\mathcal{G})$ away from each other. An $\epsilon$ of $\frac{1}{10,000}$, a gradient threshold of $10^{-10}$, and an R-Hop of 1 were provided to our SDDM solver for determining the approximate Newton direction. We compared the performance of our algorithm,  referred to SDDM-ADD hereafter, to ADD, standard gradient descent, and the exact Newton method (i.e., centralized Newton iterations). The values of the primal objective and feasibility were chosen as performance metric. 

Figure~\ref{Fig:ResBenchmark} shows these convergence metrics comparing SDDM-ADD, to ADD~\cite{c6}, standard gradient descent, and the exact Newton method (i.e., centralized Newton iteration). On relatively small networks, 30 nodes and 70 edges, our approach converges approximately an order of magnitude faster compared to both ADD and gradient descent as demonstrated in Figures~\ref{fig:PerfSM} and~\ref{fig:PerfCP}. It is also clear that on such networks, SDDM-ADD is capable of closely tracing the exact Newton method where convergence to the optimal primal objective is achieved after $\approx 200$ iterations compared to $\approx 500$ for ADD and $\approx 2000$ for gradient descent.

In the second set of experiments that goal was to evaluate the performance of SDDM-ADD on large networks where both ADD and gradient descent underperform. Results reported in Figures~\ref{fig:TrajSM} and~\ref{fig:TrajCP} on the larger 90 nodes and 200 edges network clearly demonstrate the effectiveness of our approach. Benefiting from the approximation accuracy of the Newton direction, SDDM-ADD is capable of significantly outperforming state-of-the-art methods. As shown in Figure~\ref{fig:TrajCP} convergence to the optimal solution (as computed by exact Newton iterations) is achieved after 3000 iterations, while ADD and gradient descent underperform by converging to a primal value of $10^{5}$.

\section{CONCLUSIONS}\label{Sec:Con}
In this paper we proposed a fast and accurate distributed Newton method for network flow optimization problems. Our approach utilizes the sparsity pattern of the dual Hessian to approximate the Newton direction using only local information. We achieve $\epsilon$-close approximations by proposing a novel distributed solver for symmetric diagonally dominant systems of linear equations involving M-matrices. Our solver provides a distributed implementation of the algorithm of Spielam and Peng by considering an approximate inverse chain that can be computed in a distributed fashion. 

The proposed approximate Newton method utilizes the distributed solver to obtain $\epsilon$-close approximations to the exact Newton direction up-to any arbitrary $\epsilon > 0$. We further analyzed the properties of the resulting approximate algorithm showing that, similar to conventional Newton methods, superlinear convergence within a neighborhood of the optimal value can be attained. Finally, we demonstrated the effectiveness of our method in a set of experiments on randomly generated networks. Results showed that on both small and large networks, our algorithm, outperforms state-of-the-art techniques in a variety of convergence metrics.  

Possible extensions include applications to network utility maximization~\cite{c7}, general wireless communication optimization problems~\cite{c14}, and stochastic settings~\cite{c15}.

\section*{Appendix}
The complete proofs can be found at: https://db.tt/MbBW15Zx

\end{document}